\documentclass[a4paper, 11pt]{amsart}
\usepackage{amsmath, amssymb}
\usepackage{amsfonts}
\usepackage{mathrsfs}
\usepackage[arrow,matrix,curve,cmtip,ps]{xy}

\usepackage{amsthm}

\usepackage{graphicx}
\usepackage{bm}
\usepackage{color}
\usepackage{amsfonts}
\usepackage{amscd}
\usepackage{comment}
\usepackage{epsfig}

\usepackage[colorlinks=true, pdfstartview=FitV, linkcolor=red, citecolor=blue, urlcolor=green]{hyperref}
\usepackage{tikz,etex,bbm,xspace,hyperref}
\newcommand{\arxiv}[1]{\href{http://arxiv.org/abs/#1}{\tt arXiv:\nolinkurl{#1}}}

\allowdisplaybreaks

\newtheorem{theorem}{Theorem}[section]
\newtheorem{lemma}[theorem]{Lemma}

\newtheorem*{theorem*}{Theorem}
\theoremstyle{remark}
\newtheorem{remark}[theorem]{Remark}
\newtheorem{definition}[theorem]{Definition}

\newtheorem{claim}[theorem]{Claim}
\newtheorem*{solution*}{Solution}

\numberwithin{equation}{section}

\newcommand{\ci}[1]{_{ {}_{\scriptstyle #1}}}

\newcounter{vremennyj}

\renewcommand{\eqref}[1]{Equation (\ref{#1})}

\begin{document}
\title[A new Two weight estimates for a vector-valued positive operator]{A new Two weight estimates for a vector-valued positive operator}

\author{Jingguo Lai}

\address{Department of Mathematics \\ Brown University \\ Providence, RI 02912 \\ USA}
\email{jglai@math.brown.edu}





\begin{abstract}
We give a new characterization of the two weight inequality for a vector-valued positive operator. Our characterization has a different flavor than the one of Scurry's \cite{S} and H\"{a}nninen's \cite{TH}. The proof can be essentially derived from the scalar-valued case.
\end{abstract}
\maketitle

\section{Introduction}
\subsection{Preliminaries}
\par We start with the scalar-valued positive dyadic operators. Let $\bm{\alpha} = \{\alpha_{\ci{I}}: I\in\mathcal{D}\}$ be non-negative constants associated to dyadic cubes in $\mathbb{R}^d$. Let $\mu$ and $\nu$ be weights. For a cube $I\in\mathcal{D}$, set

\begin{align}
\mathbb{E}_{\ci{I}}^\mu f := \left(\mu(I)^{-1}\int_Ifd\mu\right)\bm{1}_{\ci{I}}.
\end{align}

Consider the linear operator defined by
\begin{align}
T_{\bm{\alpha}} f := \sum_{I\in\mathcal{D}}\alpha_{\ci{I}}\cdot\mathbb{E}_{\ci{I}}^\mu f. 
\end{align}

\begin{theorem}\label{THM 1}
Let $1<p<\infty$ and let $1/p + 1/p' = 1$. $T_{\bm{\alpha}}: L^p(\mu) \rightarrow L^p(\nu)$ if and only if 
\begin{align}
\int_{J}\left|\sum_{I\in\mathcal{D}:I\subseteq J}\alpha_{\ci{I}}\cdot\bm{1}_{\ci{I}}\right|^pd\nu \leq C_{1}^p\cdot\mu(J), ~J\in\mathcal{D}\\
\int_{J}\left|\sum_{I\in\mathcal{D}:I\subseteq J}\alpha_{\ci{I}}\cdot\frac{\nu(I)}{\mu(I)}\cdot\bm{1}_{\ci{I}}\right|^{p'}d\mu \leq C_2^{p'}\cdot\nu(J), ~J\in\mathcal{D}.
\end{align}
In particular, $||T_{\bm{\alpha}}||_{\ci{L^p(\mu) \rightarrow L^p(\nu)}} \asymp C_1 + C_2$.
\end{theorem}

\par In \cite{NTV1}, Theorem \ref{THM 1}, named as bilinear embedding theorem, is proved using Bellman function technique for $p=2$. In \cite{LSU}, the case for all $1<p<\infty$ is obtained, using a technique developed in \cite{SW}. In \cite{T1} and \cite{H}, the proof is significantly simplified. In fact, the proof given in \cite{T1} also works for more general measurable spaces, see \emph{Definition} \ref{def1}. In \cite{S}, a vector-velued extension is established, and in \cite{TH}, a simplified proof is obtained using the same idea as \cite{H}.

\subsection{The main problem}
\par We are interested in the two weight estimates for the vector-valued case. Our basic setup is

\begin{definition}\label{def1}
For a measurable space $(\mathcal{X}, \mathcal{T})$, a \emph{lattice} $\mathcal{L}\subseteq\mathcal{T}$ is a collection of measurable subsets of $\mathcal{X}$ with the following properties
\begin{enumerate}
\item $\mathcal{L}$ is a union of \emph{generations} $\mathcal{L}_n, n\in\mathbb{Z}$, where each generation is a collection of disjoint measurable sets, covering $\mathcal{X}$.
\item For each $n\in\mathbb{Z}$, the covering $\mathcal{L}_{n+1}$ is a countable refinement of the covering $\mathcal{L}_n$, i.e. each set $I\in\mathcal{L}_n$ is a countable union of disjoint sets  $J\in\mathcal{L}_{n+1}$. We allow the situation where there is only one such set $J$, i.e. $J=I$; this means that $I\in\mathcal{L}_n$ also belongs to the generation $\mathcal{L}_{n+1}$.
\end{enumerate}
\end{definition}

\begin{definition}\label{def3}
For a positive measure $\mu$ on $(\mathcal{X},\mathcal{T})$ , define the \emph{averaging operator} as
\begin{align} \label{eq1}
\mathbb{E}_{\ci{I}}^\mu f :=\left(\mu(I)^{-1}\int_If d\mu\right)\mathbf{1}_{\ci{I}}.
\end{align}
\end{definition}

From now on, we assume $(\mathcal{X},\mathcal{T})$ is a measurable space, $\mathcal{L}\subseteq\mathcal{T}$ is a lattice on $\mathcal{X}$, and $\mu,\nu$ are two positive measures. 

\begin{definition}\label{def4}
Let $\bm{\alpha} = \{\alpha_{\ci{I}}: I\in\mathcal{L}\}$ be non-negative constants associated to a lattice $\mathcal{L}$ on $(\mathcal{X},\mathcal{T})$. Define a vector-valued operator
\begin{align}
\bm{T}_{\bm{\alpha}}f := \{\alpha_{\ci{I}}\cdot\mathbb{E}_{\ci{I}}^\mu f\}_{\ci{I\in\mathcal{L}}}.
\end{align}
\end{definition}

\begin{theorem}[Two weight estimates for a vector-valued positive operator]\label{thm1}
Let $1<p<\infty$ and $1\leq q<\infty$.
\begin{align}\label{eq4}
\int_{\mathcal{X}}\left[\sum_{I\in\mathcal{L}}\left|\alpha_{\ci{I}}\cdot\mathbb{E}_{\ci{I}}^\mu f\right|^q\right]^{\frac{p}{q}}d\nu \leq C^p\int_{\mathcal{X}}|f|^pd\mu, 
\end{align}
holds if and only if
\begin{enumerate}
\item  for the case $1<p\leq q$, we have
\begin{align}\label{eq5}
\int_{J}\left|\sum_{I\in\mathcal{L}:I\subseteq J}\alpha_{\ci{I}}^q\cdot\bm{1}_{\ci{I}}\right|^{\frac{p}{q}}d\nu \leq C_{1}^p\cdot\mu(J), ~J\in\mathcal{L}
\end{align}
\item for the case $q<p<\infty$, we have both (\ref{eq5}) and
\begin{align}\label{eq6}
\int_{J}\left|\sum_{I\in\mathcal{L}:I\subseteq J}\alpha_{\ci{I}}^q\cdot\frac{\nu(I)}{\mu(I)}\cdot\bm{1}_{\ci{I}}\right|^{\left(\frac{p}{q}\right)'}d\mu \leq C_2^{p'}\cdot\nu(J), ~J\in\mathcal{L}.
\end{align}
\end{enumerate}
In particular, $C \asymp C_1 + C_2$.
\end{theorem}

\begin{remark}
The case $q = 1$ is a generalization of Theorem \ref{THM 1} to the measurable space setting. The proof given in \cite{T1} adapts to this general situation.
\end{remark}

\begin{remark}
In \cite{S} and \cite{TH}, to obtain the two testing conditions, they first rewrite (\ref{eq4}) into
\begin{align}
\sum_{I\in\mathcal{L}}\alpha_{\ci{I}}\cdot\mathbb{E}_{\ci{I}}^\mu f\cdot\mathbb{E}_{\ci{I}}^\nu g_{\ci{I}}\cdot\nu(I) \lesssim ||f||_{\ci{L^p(\mu)}}\cdot||\{g_{\ci{I}}\}_{\ci{I\in\mathcal{L}}}||_{\ci{L^{p'}(l^q, \nu)}}.
\end{align}
Setting $f=\textbf{1}_{\ci{J}}$, one deduces (\ref{eq5}). For the second testing condition, one turns to consider the family of functions $\{g_{\ci{I}}\}_{\ci{I\in\mathcal{L}}}$ supported on $J\in\mathcal{L}$ with $l^q$-norm equal to 1. This gives
\begin{align}\label{eq many}
\int_J\left|\sum_{I\in\mathcal{L}:I\subseteq J}\alpha_{\ci{I}}\cdot\frac{\nu(I)}{\mu(I)}\cdot\mathbb{E}_{\ci{I}}^\nu g_{\ci{I}}\right|^{p'}d\mu\lesssim \nu(J), ~J\in\mathcal{L}.
\end{align}
Compare Theorem \ref{thm1} with the main results in \cite{S} and \cite{TH}. We have a very different condition (\ref{eq6}) than (\ref{eq many}) with seemingly \emph{'wrong'} exponents. However, We will see that both (\ref{eq5}) and (\ref{eq6}) are testing conditions on some families of special functions.
\end{remark}

\section{The case: $1<p\leq q$}
\par We will see in this section that when $1<p\leq q$, (\ref{eq5}) is equivalent to (\ref{eq4}). On one hand, (\ref{eq5}) can be deduced from (\ref{eq4}) by setting $f=\textbf{1}_{\ci{J}}$. On the other hand,  consider the maximal function
\begin{align}\label{eq M}
M_\mu f(x) := \sup_{x\in I, I\in\mathcal{L}}\left|\mathbb{E}_{I}^\mu f(x)\right|.
\end{align}
The celebrated \emph{Doob's martingale inequality} asserts
\begin{align}\label{eq Doob}
||M_\mu f||_{\ci{L^p(\mu)}} \leq p'\cdot||f||_{\ci{L^p(\mu)}}.
\end{align}
Let $E_k := \{x\in\mathcal{X}:M_\mu f(x) > 2^k\}$ and let $\mathcal{E}_k := \{I\in\mathcal{L}:I\in E_k\}$. Note that $E_k$ is a disjoint union of maximal sets in $\mathcal{E}_k$, maximal in the sense of inclusion. Denote these disjoint maximal sets by $\mathcal{E}^*_k$. Hence, $E_k = \sqcup_{\ci{J\in\mathcal{E}^*_k}} J$.

\begin{align*}
\int_{\mathcal{X}}\left[\sum_{I\in\mathcal{L}}\left|\alpha_{\ci{I}}\cdot\mathbb{E}_{\ci{I}}^\mu f\right|^q\right]^{\frac{p}{q}}d\nu
& \leq \sum_k\int_{E_{k}}\left[\sum_{I\in\mathcal{E}_k\setminus\mathcal{E}_{k+1}}\left|\alpha_{\ci{I}}\cdot\mathbb{E}_{\ci{I}}^\mu f\right|^q\right]^{\frac{p}{q}}d\nu, ~1<p\leq q\\
& \leq \sum_k2^{(k+1)p}\int_{E_k}\left[\sum_{I\in\mathcal{E}_k\setminus\mathcal{E}_{k+1}}\alpha_{\ci{I}}^q\cdot\bm{1}_{\ci{I}}\right]^{\frac{p}{q}}d\nu\\
& \leq \sum_k2^{(k+1)p}\sum_{J\in\mathcal{E}^*_k}\int_{J}\left[\sum_{I\in\mathcal{L}:I\subseteq J}\alpha_{\ci{I}}^q\cdot\bm{1}_{\ci{I}}\right]^{\frac{p}{q}}d\nu\\
& \leq C_1^p\cdot\sum_k2^{(k+1)p}\cdot\mu(E_k), ~(\ref{eq5})\\
& \lesssim C_1^p\cdot ||M_\mu f||_{\ci{L^p(\mu)}}^p\\
& \leq C_1^p\cdot(p')^p\cdot||f||_{\ci{L^p(\mu)}}^p, ~(\ref{eq Doob}).
\end{align*}

\section{The case: $q<p<\infty$, a counterexample}
\par In this section, we see that (\ref{eq5}) itself is not sufficient for (\ref{eq4}) for the case $q<p<\infty$.

\par Consider the real line $\mathbb{R}$ with the Borel $\sigma$-algebra $\mathcal{B}(\mathbb{R})$. Let the lattice be all the tri-adic intervals. We specify the positive measures $\mu, \nu$, the non-negative constants $\bm{\alpha} = \{\alpha_{\ci{I}}: I\in\mathcal{L}\}$, and the functions $f$ in the following way. 

\par Let $C=\cap_{n \geq 0}C_n$ be the $1/3$-Cantor set, where $C_0=[0,1), C_1=[0, 1/3)\cup[2/3, 1)$ and, in general, $C_n=\cup\left\{[x, x+3^{-n}): x=\sum_{j=1}^n\varepsilon_j3^{-j}, \varepsilon_j\in\{0, 2\}\right\}$. 

\begin{enumerate}
\item The measure $\mu$ is the Lebsgue measure restricted on $[0, 1)$ and the measure $\nu$ is the Cantor measure, i.e. $\nu(I)=2^{-n}$ for each $I$ belongs to a connect component of $C_n$.

\item Define $\alpha_{\ci{I}} = (2/3)^{n/p}$ for each $I$ belongs to a connect component of $C_n$.

\item For the function $f$, consider the gap of $C$, i.e. $[0, 1)\setminus C$. This is a disjoint union of tri-adic intervals. Let $f=(3/2)^{n/p}\cdot n^{-r}$ for each $I\in [0, 1)\setminus C$ with length of $I$ equals $3^{-n}$, where $r$ is to be chosen later in the proof.
\end{enumerate}

\begin{claim}
The above construction gives a counterexample with properly chosen $r$.
\end{claim}

\begin{proof}
We begin with checking (\ref{eq5}). It suffices to check for every $J$ belongs to a connected component of $C_n$, and thus $\mu(J)=3^{-n}$. Note that

\begin{align*}
\left|\sum_{I\in\mathcal{L}:I\subseteq J}\alpha_{\ci{I}}^q\cdot\bm{1}_{\ci{I}}\right|^{\frac{p}{q}} \leq \left|\sum_{k\geq n}\left(\frac{2}{3}\right)^{\frac{qk}{p}}\right|^{\frac{p}{q}}\asymp\left(\frac{2}{3}\right)^{n}.
\end{align*}

\par Hence,
\begin{align*}
\int_{J}\left|\sum_{I\in\mathcal{L}:I\subseteq J}\alpha_{\ci{I}}^q\cdot\bm{1}_{\ci{I}}\right|^{\frac{p}{q}}d\nu \lesssim \left(\frac{2}{3}\right)^{n}\cdot\nu(J)=\mu(J).
\end{align*}

Next, we show that (\ref{eq4}) fails. This requires a smart choice of $r$ in the definition of $f$. Picking $r > \frac{1}{p}$, we have

\begin{align*}
||f||_{\ci{L^p(\mu)}}^p= \int_0^1|f|^pdx
=\sum_{n\geq1}\left(\frac{3}{2}\right)^nn^{-pr}\cdot\frac{1}{3^n}\cdot2^n
=\sum_{n\geq1}n^{-pr}<\infty.
\end{align*}

Since $q<p<\infty$, we can pick $r$ such that $\frac{1}{p}<r<\frac{1}{q}$. Note that for every $I$ belongs to a connected component of $C_n$, we have

\begin{align*}
\mathbb{E}_{\ci{I}}^\mu f \geq \frac{1}{3}\left(\frac{3}{2}\right)^{\frac{n+1}{p}}(n+1)^{-r}.
\end{align*}

Hence, consider $\mathcal{I}_n=\left\{I:I ~\textup{is tri-adic with length less than or equal to}~ 3^{-n}\right\}$, 

\begin{align*}
\sum_{I\in\mathcal{I}_n}\left|\alpha_{\ci{I}}\cdot\mathbb{E}_{\ci{I}}^\mu f\right|^q\cdot\textbf{1}_{\ci{C_n}}
\geq \sum_{k\leq n}\left|\frac{1}{3}\left(\frac{3}{2}\right)^{\frac{1}{p}}(k+1)^{-r}\right|^q
\gtrsim \sum_{k\leq n}(k+1)^{-qr}.
\end{align*}

And so,
\begin{align*}
\int\left[\sum_{I\in\mathcal{I}_n}\left|\alpha_{\ci{I}}\cdot\mathbb{E}_{\ci{I}}^\mu f\right|^q\right]^{\frac{p}{q}}d\nu
\gtrsim \left[\sum_{k\leq n}(k+1)^{-qr}\right]^{\frac{p}{q}}\cdot\nu(C_n)\rightarrow\infty~\textup{as}~ n\rightarrow\infty.
\end{align*}
We can see that the condition $q<p<\infty$ is crutial in our construction.
\end{proof}

\section{The case: $q<p<\infty$}
\par We discuss the case $q<p<\infty$ of Theorem \ref{thm1} in this section. In particular, we see that both (\ref{eq5}) and (\ref{eq6}) are testing conditions on some families of special functions. 
\par To begin, since

\begin{align}
||\bm{T}_{\bm{\alpha}}f||_{\ci{L^p(l^q, \nu)}}^q =\sup_{||g||_{\ci{L^{\left(p/q\right)'}(\nu)}} = 1} \int_{\mathcal{X}}\left[\sum_{I\in\mathcal{L}}\left|\alpha_{\ci{I}}\cdot\mathbb{E}_{\ci{I}}^\mu f\right|^q\right]gd\nu,
\end{align}

we can write

\begin{align}\label{eq last}
||\bm{T}_{\bm{\alpha}}||_{\ci{L^p(\mu)\rightarrow L^p(l^q, \nu)}}^q =\sup_{||f||_{\ci{L^p(\mu)}=1}} ~\sup_{||g||_{\ci{L^{\left(p/q\right)'}(\nu)}} = 1} \int_{\mathcal{X}}\left[\sum_{I\in\mathcal{L}}\left|\alpha_{\ci{I}}\cdot\mathbb{E}_{\ci{I}}^\mu f\right|^q\right]gd\nu. 
\end{align}

Without loss of generality, we assume that both $f$ and $g$ are non-negative. The following lemma reduces us to the scalar-valued case.

\begin{lemma}
\begin{align}\label{eq essential}
||\bm{T}_{\bm{\alpha}}||_{\ci{L^p(\mu)\rightarrow L^p(l^q, \nu)}}^q \asymp \sup_{||f||_{\ci{L^p(\mu)}=1}}~  \sup_{||g||_{\ci{L^{\left(p/q\right)'}(\nu)} }= 1} \int_{\mathcal{X}}\left[\sum_{I\in\mathcal{L}}\alpha_{\ci{I}}^q\cdot\mathbb{E}_{\ci{I}}^\mu(f^q)\right]gd\nu. 
\end{align}
\end{lemma}

\par An easy application of H\"{o}lder's inequality shows that the LHS of (\ref{eq essential}) is no more than its RHS. The other half of this lemma depends on the following famous \emph{Rubio de Francia Algorithm}.

\begin{lemma}[Rubio de Francia Algorithm] \label{L1}
For every $q<p<\infty$ and $f\in L^p(\mu)$, there exists a function $F\in L^p(\mu)$, such that $f\leq F$, $||F||_{\ci{L^p(\mu)}}\asymp||f||_{\ci{L^p(\mu)}}$ and 
$$\mu(I)^{-1}\int_IF^qd\mu\lesssim\inf_{\ci{x\in I}}~F^q(x),  ~~~I\in\mathcal{L}.$$

\end{lemma}

\begin{proof}
\par Consider the maximal operator $M_\mu$ defined in (\ref{eq M}). Doob's martingale inequality (\ref{eq Doob}) implies
 
\begin{align}
||M_\mu||_{\ci{L^{p/q}(\mu)\rightarrow L^{p/q}(\mu)}} \leq \left(\frac{p}{q}\right)'.
\end{align} 

Denote $M_\mu^{(0)}=Id$, $M_\mu^{(1)}=M_\mu$ and $M_\mu^{(k)}=M_\mu\circ M_\mu^{(k-1)}$. Define the function $F$ by

\begin{align}
F=\left[\sum_{k\geq0}\left(2||M_\mu||_{\ci{L^{p/q}(\mu)\rightarrow L^{p/q}(\mu)}}\right)^{-k}M_\mu^{(k)}\left(f^q\right)\right]^{\frac{1}{q}}.
\end{align}

\par First we check the validity of the definition for $F$. Note that

\begin{align*}
||F||_{\ci{L^p(\mu)}}^q
& = \left\{\int_{\mathcal{X}} \left[\sum_{k\geq0}\left(2||M_\mu||_{\ci{L^{p/q}(\mu)\rightarrow L^{p/q}(\mu)}}\right)^{-k}M_\mu^{(k)}\left(f^q\right)\right]^{\frac{p}{q}}d\mu\right\}^{\frac{q}{p}}\\
& \leq \sum_{k\geq0}\left(2||M_\mu||_{\ci{L^{p/q}(\mu)\rightarrow L^{p/q}(\mu)}}\right)^{-k}\left(\int_{\mathcal{X}} \left|M_\mu^{(k)}\left(f^q\right)\right|^{\frac{p}{q}}d\mu\right)^{\frac{q}{p}},~~~~~\textup{Minkowski inequality}\\
& \leq \sum_{k\geq0}\left(2||M_\mu||_{\ci{L^{p/q}(\mu)\rightarrow L^{p/q}(\mu)}}\right)^{-k}\left(||M_\mu||_{\ci{L^{p/q}(\mu)\rightarrow L^{p/q}(\mu)}}\right)^{k}||f||_{\ci{L^p(\mu)}}^q=2||f||_{\ci{L^p(\mu)}}^q.
\end{align*}

\par Hence, $F$ is the $L^{p/q}(\mu)$-limit of the partial sums and thus well-defined. Moreover, we have also proved that $||F||_{\ci{L^p(\mu)}}\lesssim||f||_{\ci{L^p(\mu)}}$.

\par Considering only $k=0$ in the definition for $F$, we have $F\geq f$. And so $||F||_{\ci{L^p(\mu)}}\asymp||f||_{\ci{L^p(\mu)}}$. Finally, note that 

\begin{align}
\mu(I)^{-1}\int_IF^qd\mu \leq \inf_{x\in I} ~M_\mu(F^q)(x)
\end{align}
and 
$$M_\mu(F^q)=\sum_{k\geq0}\left(2||M_\mu||_{\ci{L^{p/q}(\mu)\rightarrow L^{p/q}(\mu)}}\right)^{-k}M_\mu^{(k+1)}\left(f^q\right)=2||M_\mu||_{\ci{L^{p/q}(\mu)\rightarrow L^{p/q}(\mu)}}\left(F^q-f^q\right)\lesssim F^q.$$ 
Therefore, we deduce
$$\mu(I)^{-1}\int_IF^qd\mu\lesssim\inf_{\ci{x\in I}}~F^q(x),  ~~~I\in\mathcal{L}.$$
\end{proof}

Applying Rubio de Francia Algorithm, we obtain

\begin{align*}
\int_{\mathcal{X}}\left[\sum_{I\in\mathcal{L}}\alpha_{\ci{I}}^q\cdot\mathbb{E}_{\ci{I}}^\mu(f^q)\right]gd\nu
& \leq \int_{\mathcal{X}}\left[\sum_{I\in\mathcal{L}}\alpha_{\ci{I}}^q\cdot\mathbb{E}_{\ci{I}}^\mu(F^q)\right]gd\nu\\
& \lesssim \int_{\mathcal{X}}\left[\sum_{I\in\mathcal{L}}\alpha_{\ci{I}}^q\cdot\left(\mathbb{E}_{\ci{I}}^\mu(F)\right)^q\right]gd\nu\\
&\leq ||\bm{T}_{\bm{\alpha}}||_{\ci{L^p(\mu)\rightarrow L^p(l^q, \nu)}}^q\cdot||F||_{\ci{L^p(\mu)}}\cdot ||g||_{\ci{L^{\left(p/q\right)'}(\nu)}},~ (\ref{eq last})\\
& \lesssim ||\bm{T}_{\bm{\alpha}}||_{\ci{L^p(\mu)\rightarrow L^p(l^q, \nu)}}^q, ~\left(||f||_{\ci{L^p(\mu)}}=||g||_{\ci{L^{\left(p/q\right)'}(\nu)}}=1\right).
\end{align*}

Now that our problem is reduced to determine a necessary and sufficient condition of
\begin{align}
\int_{\mathcal{X}}\left|\sum_{I\in\mathcal{L}}\alpha_{\ci{I}}^q\cdot\mathbb{E}_{\ci{I}}^\mu(f)\right|^{\frac{p}{q}}d\nu \lesssim \int_{\mathcal{X}}|f|^{\frac{p}{q}}d\mu,
\end{align}
we may consult to the scalar-valued Theorem \ref{THM 1}. Note that Theorem \ref{THM 1} still holds in the measurable space setting as is pointed out in \cite{T1}. Therefore, Theorem \ref{thm1} follows from Theorem \ref{THM 1} for free, and both (\ref{eq5}) and (\ref{eq6}) are testing conditions with respect to the derived scalar-valued problem. 

\section*{acknowledgement}
\par The author would like to thank his PhD thesis advisor, Serguei Treil, for many enlightening and insightful discussions on this problem.

\end{document}